\documentclass[11pt,reqno]{amsart}
\usepackage{amsmath,amssymb,amsthm,amsfonts,amscd,color,flafter,graphicx,verbatim,pinlabel,mathrsfs,setspace}
\usepackage[colorlinks=false]{hyperref}
\hypersetup{backref,colorlinks=true,citecolor=blue,linkcolor=blue}
\usepackage{tikz-cd}
\usepackage[none]{hyphenat}
\title{Sutured ECH and Contact 2-Handles}
\author{Yen-Lin Chen}
\address{Department of Mathematics \\ University at Buffalo \\ Buffalo, NY 14260}
\email{yenlinch@buffalo.edu}
\date{\today}

\newtheorem{thm}{Theorem}[section]

\newtheorem{corollary}[thm]{Corollary}
\newtheorem{prop}[thm]{Proposition}
\newtheorem{lemma}[thm]{Lemma}

\theoremstyle{definition}
\newtheorem{Def}[thm]{Definition}
\newtheorem{Ex}[thm]{Example}

\newcommand\isomto{\xrightarrow{\sim}}
\newcommand\ecc{\mathit{ECC}}
\newcommand\ech{\text{ECH}}
\newcommand\HMfrom{\widehat{\mathit{HM}}}
\begin{document}
\bibliographystyle{plain}
\maketitle
\sloppy
\begin{abstract}
We show that there are well-defined maps on sutured ECH induced by contact 2-handle attachments and that the sutured ECH contact class is functorial under such maps.
\end{abstract}
\tableofcontents

\section{Introduction}
\label{sec:introduction}
In \cite{bs2}, Baldwin and Sivek studied the effect of contact handle attachments on Kronheimer and Mrowka's sutured monopole homology \cite{km-excision} and proved an analogue of Honda's bypass exact triangle using \"{O}zba\u{g}cı's description of a bypass attachment as a contact 1- and 2-handle pair attachment that cancel topologically \cite{Ozbagci_bypass}. The aim of this article is to establish the corresponding contact handle attachment maps for Colin--Ghiggini--Honda--Hutchings's sutured embedded contact homology (ECH). The latter has recently been shown to be isomorphic to Juh\'asz's sutured Floer homology \cite{juhasz-sutured} by Colin, Ghiggini, and Honda \cite{cgh-sutured-HF=ECH}.

It was shown in \cite{kst} that there is a natural model for contact 1-handles that works well with sutured ECH, and that attaching a contact 1-handle to a sutured contact 3-manifold induces a canonical isomorphism between sutured ECH of those  contact manifolds before and after the handle attachment. In contrast, a model for contact 2-handles to work with sutured ECH has been elusive. Therefore we adopt an ad hoc approach that requires investigating the Reeb dynamics as we switch between (smooth) convex and sutured boundary conditions. More specifically, given a sutured contact 3-manifold $(M,\Gamma,\alpha)$ and a simple closed curve $\gamma$ on $\partial M$ that intersects $\Gamma$ in exactly two points, we first round corners of $M$ to turn it into a contact 3-manifold with convex boundary, then attach a contact 2-handle along $\gamma$ following Giroux \cite{giroux-convex}, and finally turn the resulting contact 3-manifold with convex boundary into a sutured contact 3-manifold $(M_{\gamma} ,\Gamma',\alpha_{\gamma})$ following a process that takes place in a small neighborhood of the dividing set. With the preceding understood, our main result can be stated as follows:
\begin{thm}
\label{thm:main}
    Let $(M,\Gamma,\alpha)$ be a sutured contact 3-manifold and $\gamma$ be a simple closed curve on its boundary that intersects the suture in exactly two points. Then there is a well-defined map 
    \[
    F_{h_2} : 
    \ech(M,\Gamma,\alpha) 
    \rightarrow
    \ech(M_{\gamma} ,\Gamma',\alpha_{\gamma})
    \]
    that sends the empty orbit set to the empty orbit set.
\end{thm}
\noindent In upcoming work in progress, we prove an analogue of the bypass exact triangle for sutured ECH and discuss some applications.

The organization of this article is as follows. In Section \ref{sec:preliminaries}, we give a quick review of $\ech$ and its sutured version as well as a brief account of the convex surface theory. We refer the reader to \cite{hutchings-budapest} for a more detailed exposition of $\ech$. Also in that section, we show that going from sutured to smooth convex boundary condition and back does not change $\ech$ and we flesh out contact 2-handle attachment in our context following Giroux's recipe \cite{giroux-convex}. In Section \ref{sec:contact_2-handles}, we describe the prototypical contact 2-handle attachment map on sutured ECH and defer some details of the construction to Section \ref{sec:proof33}. Section \ref{sec:suturingvshandle} proves that attaching a contact 1-handle before or after switching from smooth convex to sutured boundary condition does not change sutured ECH.  
\subsection*{Acknowledgements}
The author was supported in part by a Simons Foundation grant No. 519352 and he is very grateful to his advisor \c{C}a\u{g}atay Kutluhan for suggesting this problem, for posing and pushing the questions, for enduring the author's rambles, rewrites and stalls and for meticulously reading and revising an early draft. The author also thanks Agniva Roy and Yuan Yao for helpful conversations.
\section{Preliminaries}
\label{sec:preliminaries}
\subsection{Definitions and conventions}
\subsubsection{Sutured contact 3-Manifolds}
A convex sutured contact 3-manifold is a balanced sutured 3-manifold equipped with a contact structure co-oriented by an adapted contact form. We include the Definitions 2.1 and 2.2 from \cite{cgh-sutured-HF=ECH} for completeness. (cf. \cite[Definitions 2.7 and 2.8]{cghh}.)
\begin{Def}
    A balanced \emph{sutured 3-manifold} is a triple $(M , \Gamma , U (\Gamma))$, where $M$ is a compact 3-manifold with boundary and corners, $\Gamma$ is an oriented 1-manifold in $\partial M$ called the suture, and $U(\Gamma) \simeq [-1 , 0] \times \Gamma \times [-1 , 1]$ is a neighborhood of $\Gamma \simeq \{ 0 \} \times \Gamma \times \{ 0 \}$ in $M$ with coordinates $(\tau , t) \in [-1, 0] \times [-1,1]$, such that the following hold
    \begin{itemize}\leftskip-0.25in
        \item $M$ has no closed components, 
        \item $U(\Gamma) \cap \partial M \simeq (\{ 0 \} \times \Gamma \times [-1,1]) \cup ([-1 , 0] \times \Gamma \times \{ -1 \}) \cup ([-1 , 0] \times \Gamma \times \{ 1 \})$,
        \item 
        $\partial M \setminus (\{ 0 \} \times \Gamma \times (-1,1))$ is the disjoint union of two submanifolds which we call $R _ - (\Gamma)$ and $R _ + (\Gamma)$, where the orientation of $\partial M$ agrees with that of $R _ + (\Gamma)$ and is opposite that of $R _ - (\Gamma)$, and the orientation of $\Gamma$ agrees with the boundary orientation of $R _ + (\Gamma)$,
        \item 
        the corners of $M$ are precisely $\{ 0 \} \times \Gamma \times \{ \pm 1 \}$,
        \item 
        $R _ {\pm} (\Gamma)$ have no closed components and $\chi (R _ - (\Gamma)) = \chi (R _ + (\Gamma))$.
    \end{itemize}
\end{Def}

\begin{Def}
    Let $(M , \Gamma , U (\Gamma))$ be a balanced sutured 3-manifold. Then $(M , \Gamma , U (\Gamma) , \xi)$ is a \emph{sutured contact manifold} if there exists a contact form $\alpha$ for $\xi$ with Reeb vector field $R _ {\alpha}$ such that:
    \begin{itemize}\leftskip-0.25in
        \item $R _ {\alpha}$ is positively transverse to $R _ + (\Gamma)$ and negatively transverse to $R _ - (\Gamma)$,
        \item 
        $\alpha = Cdt + \beta$ on $U (\Gamma)$ for some constant $C > 0$, where $\beta$ is independent of $t$. In particular, $R _ {\alpha} = \frac{1}{C} \partial _ t$ on $U (\Gamma)$.
    \end{itemize}
\end{Def}

Note that the following are immediate consequences of the above definitions:
\begin{itemize}\leftskip-0.25in
    \item 
    The boundary $\partial M$ comes in three components: $R _ {\pm} (\Gamma)$, $\{ 0 \} \times \Gamma \times [-1 ,1]$ which we shall refer to as the upper/lower and vertical boundaries of $M$, denoted by ${\partial _ {h}}^\pm M$ and $\partial _ {v} M$, respectively.
    \item 
    $(R _ {\pm} (\Gamma) , \beta _ {\pm} = \alpha | _ {R _ {\pm} (\Gamma)})$ are Liouville manifolds (see \cite[Definition 2.8]{cghh}) with respective Liouville vector fields restricting to $\partial _ {\tau}$ over $R _ {\pm} (\Gamma) \cap U(\Gamma)$.
    \item 
    On $U(\Gamma) \simeq [-1 , 0] _ {\tau} \times \Gamma _ \theta \times [-1 ,1] _ t$,
    \begin{equation}
    \label{alpha_over_vert}
        \alpha = C dt + e ^ {\tau} \beta_0
    \end{equation}
     where $\beta_0$ is a volume form on $\Gamma$.
    \item 
    The $t$-coordinate function extends to collar neighborhoods of ${\partial _ {h}}^\pm M$ by using the flow of $C R _ {\alpha}$: $(1 - \varepsilon , 1] _ t \times R _ + (\Gamma)$, $[-1 , -1 + \varepsilon) _ t \times R _ - (\Gamma)$ over which 
    \begin{equation}
    \label{alpha_over_pm}
        \alpha = Cdt + \beta _ {\pm}
    \end{equation}
     respectively.
\end{itemize}
\begin{Ex}{\cite[Example 2.9]{cghh}}
    Let $(S , \beta)$ be a 2-dimensional Liouville domain with a collar neighborhood $\partial S \times [-1 , 0] _ {\tau}$ of $\partial S$ afforded by the Liouville field identified with $\partial _ {\tau}$ then
    \[
    (S \times [-1 , 1] _ t , \partial S \times \{ 0 \} ,  \partial S \times [-1 , 0] _ {\tau} \times [-1 ,1 ] _ t , dt + \beta)
    \]
    is a sutured contact 3-manifold.
\end{Ex}

\subsubsection{Sutured Embedded Contact Homology}
Here we give a brief overview of Hutchings's embedded contact homology (ECH) and its sutured version as defined by Colin, Ghiggini, Honda, and Hutchings. We refer the reader to \cite{hutchings-budapest} for details on the definition and structure of ECH.

Given a closed contact three manifold $Y$ equipped with a nondegenerate contact form $\lambda$ and some fixed $\Gamma \in H _ 1 (Y)$, the embedded contact homology $\ech (Y , \lambda , \Gamma ; J)$ is the homology of a chain complex freely generated by admissible Reeb orbit sets $\{ (\alpha _ i , m _ i) \}$ representing $\Gamma$, whose differential counts $\ech$ index $1$ holomorphic curves asymptotic at its positive/negative ends to such orbit sets in $\mathbb{R} \times Y$ equipped with a generic $\mathbb{R}$-invariant symplectization-admissible almost complex structure $J$. 

Although its definition relies on particular choices of contact form $\lambda$ and almost complex structure $J$, it was shown by Taubes in \cite{taubes1} that there is a canonical isomorphism between ECH and Seiberg--Witten Floer cohomology:
\begin{equation*}
    \ech _ {*} 
    (Y , \lambda , \Gamma ; J) 
    \simeq
    \HMfrom ^ {- *}
    (Y , \mathfrak{s} _ {\xi} + PD(\Gamma)).
\end{equation*}
Consequently, $\ech$ is in fact a topological invariant. In contrast, filtered $\ech$, $\ech ^ L (Y , \lambda , \Gamma , J)$, which is defined by those orbit sets $\alpha = \{ ( \alpha _ i , m _ i)\}$ with symplectic action
\[
\mathcal{A} (\alpha)
=
\int _ {\alpha} \lambda 
= \sum _ i m _ i \int _ {\alpha _ i} \lambda
\]
less than some fixed $L>0$ is sensitive to the choice of the contact form $\lambda$. Filtered $\ech$ was shown to be independent of a choice of almost complex structure in \cite{ht2}; hence it defines canonical groups $\ech ^ L _ {*} (Y , \lambda , \Gamma)$ that form a directed system such that 
\[
\lim 
_ {L \rightarrow \infty}
\ech ^ L (Y , \lambda, \Gamma) 
= 
\ech (Y , \lambda, \Gamma).
\]
When there is no need to specify the homology class $\Gamma$, one can work with
\[
\ech ^ L _ * (Y , \lambda) 
:= 
\bigoplus _ {\Gamma \in H _ 1 (Y)}
\ech ^ L _ * (Y , \lambda , \Gamma),
\]
which satisfy
\[
\lim 
_ {L \rightarrow \infty}
\ech ^ L (Y , \lambda) 
= 
\ech (Y , \lambda).
\]
Moreover, given an exact symplectic cobordism $(X , \omega = d \lambda)$ from $(Y _ +, \lambda _ +)$ to $(Y _ -, \lambda _ -)$ and $L>0$, there is a cobordism induced map
\[
\Phi ^ L (X , \omega) :
\ech ^ L (Y _ + , \lambda _ +)
\rightarrow
\ech ^ L (Y _ - , \lambda _ -)
\]
satisfying the properties listed in \cite[Theorem 1.9]{ht2}. 

For a sutured contact 3-manifold $(M , \Gamma , U (\Gamma) , \alpha)$, its sutured ECH is defined in \cite{cghh} to be the homology of the chain complex $\ecc(M, \Gamma, U(\Gamma), \alpha)$ generated by admissible orbit sets that lie in the interior of $M$ and with differential defined by counting ECH index 1 holomorphic curves in $\mathbb{R} \times M ^ *$ where \[M^\ast = M _ v 
    \cup ([0 , \infty) _ \tau \times \Gamma \times \mathbb{R} _ t ),\] with 
\begin{itemize}\leftskip-0.25in
    \item 
    $M _ v$ a vertical extension of $(M , \Gamma , U(\Gamma) , \alpha )$ by
    $(R _\pm(\Gamma) \times [1 , \infty) _ t , Cdt + \beta _ \pm)$ attached along $R _ {\pm} (\Gamma)$, respectively,
    \item 
    $([0 , \infty) _ {\tau} \times \Gamma \times \mathbb{R} _ t , Cdt + e ^ {\tau} \beta_0 )$ is attached to $M _ v$ along the infinite cylinder boundary $\{ 0 \} _ {\tau} \times \Gamma \times \mathbb{R} _ t$,
\end{itemize}
and $\mathbb{R} \times M ^ *$ is endowed with an almost complex structure tailored to $(M ^ * , \alpha ^ *)$. (See \cite[Section 3.1]{cghh} for definition of a tailored almost complex structure and \cite[Section 5]{cghh} as to why working with such an almost complex structure is necessary.) The resulting homology group is denoted by $\ech (M , \Gamma , \alpha , J)$. 

Like ECH, its sutured version also admits a filtration by symplectic action. For $L>0$, filtered sutured ECH groups are denoted by $\ech^L(M , \Gamma , \alpha , J)$, and 
\[
\lim 
_ {L \rightarrow \infty}
\ech ^ L (M , \Gamma , \alpha , J) 
= 
\ech (M , \Gamma , \alpha , J).
\]
It was shown in \cite{cgh-openbook}, and independently in \cite{kst}, that filtered sutured ECH depends only on the underlying sutured 3-manifold and the contact structure. In fact, it was shown in \cite{kst} that sutured ECH assigns a canonically defined group to a sutured contact 3-manifold that depends a priori only on the restriction of an adapted contact form on the boundary. More recently, Colin, Ghiggini, and Honda have shown that sutured ECH is isomorphic to sutured Floer homology of Juh\'asz \cite{juhasz-sutured}, settling \cite[Conjecture 1.5]{cghh}. Therefore, sutured ECH is a topological invariant of the underlying sutured 3-manifold. 

\subsubsection{Convex hypersurfaces in contact manifolds}

\begin{Def}
    \label{def:convex_surface}
    A hypersurface $\Sigma$ within a $(2n + 1)-$dimensional contact manifold $(M , \xi = \ker \alpha)$ is said to be convex if there is a contact vector field $X$ defined in a tubular neighborhood $N(\Sigma)$ of $\Sigma$ and is transverse to $\Sigma$, i.e. $\mathcal{L} _ X \alpha = \mu \alpha$ for some $\mu \in C ^ {\infty} (N(\Sigma))$ and $X_p\notin T_p\Sigma$ for $p\in\Sigma$ (see \cite[Lemma 1.5.8 (b)]{geiges_intro}.)
\end{Def}

    Convexity is a local condition since $X$ is generated by a contact Hamiltonian function (see \cite[Theorem 2.3.1]{geiges_intro}.) Having fixed $X$, the dividing set $\Gamma$ on $\Sigma$, defined to be $\{ x \in \Sigma \ | \ X (x) \in \xi (x)\}$, divides $\Sigma$ into two subsurfaces denoted by $\Sigma_\pm$ over which $X$ is positively/negatively transverse to $\xi$, respectively. 
    
    Using the flow of $X$, one can identify a collar neighborhood of $\Sigma$ with $\mathbb{R} _ s \times \Sigma$ which we shall refer to as a vertically invariant neighborhood of $\Sigma$. On this neighborhood, we can write 
    \begin{equation}
    \label{eq:convex_form}
        \alpha = m (f ds + \beta)
    \end{equation}
    where 
    \begin{enumerate}\leftskip-0.25in
        \item $m = \exp (-\int _ {0} ^ {s} \mu(\zeta,p) d\zeta)$ where $\mathcal{L} _X \alpha = \mu \alpha$ (cf. \cite[Section 4.6.2]{geiges_intro}),
        \item $\beta \in \Omega ^ 1 (\Sigma)$ and $f \in C ^ \infty (\Sigma)$ both of which are independent of $s$. 
    \end{enumerate}
    Thus every convex hypersurface $\Sigma$ admits a collar neighborhood that is contactomorphic to $(\mathbb{R} _ s \times \Sigma , fds + \beta)$.
    
    With the above identification understood, the dividing set $\Gamma$ is cut out by the zero locus of $f$, i.e. $\{ f = 0\} \subset \Sigma$ whereas $\Sigma_\pm = \{ \pm f > 0\}$. Using the contact condition 
    \[
    \alpha \wedge (d \alpha) ^ n
    =
    f ds (d \beta) ^ n 
    +
    n df ds \beta (d \beta) ^ {n-1} >0,
    \]
    one can see that $df \neq 0$ along $\Gamma$ and hence $\Gamma$ is a codimension-1 submanifold of $\Sigma$ which in the three dimensional case, $n=1$, amounts to a disjoint union of embedded circles.  

\subsection{From sutured to convex boundary and back}
\label{ss:sutured-convex}
In order to describe the effect of contact 2-handle attachment on sutured ECH, we adopt Baldwin and Sivek's strategy in \cite{bs2} that interprets contact 2-handle attachment on a contact 3-manifold with convex boundary as performing contact surgery along a parallel copy of the attaching curve within the interior of a vertically invariant neighborhood of the boundary. To do so, we first switch from sutured to convex boundary condition, then attach a contact 2-handle on the resulting contact 3-manifold with convex boundary following Giroux \cite{giroux-convex}, and finally switch back from convex to sutured boundary condition. 

To begin, we recall how to switch from sutured to convex boundary condition as is described in \cite[Section 4.1]{cghh}. Given a sutured contact 3-manifold $(M,\Gamma, U(\Gamma),\mathit{ker}(\alpha))$, the latter amounts to rounding its corners and it takes place entirely within the neighborhood $U(\Gamma)$ of its vertical boundary parametrized by
\begin{align*}
    [-1 , 1] _ t \times (-1 , 0 ] _ \tau \times \Gamma _ \theta,
\end{align*}
To be precise, we describe the new boundary after rounding corners in terms of the graph of a convex even function $\Xi:[-1,1]\to\mathbb{R}$ that is smooth on $(-1,1)$ such that $\Xi(0) = 0$, $\dot{\Xi}(0)=0$, $\dot{\Xi}\neq0$ for $|t|>1-\epsilon$ with $\epsilon\in(0,\frac{1}{100})$, and $\lim_{t\to\pm1} \dot{\Xi}(t)= \mp \infty$, so that points on the new boundary are defined by 
\begin{align*}
    (t , \tau , \theta) = (t , \Xi(t) , \theta).
\end{align*}

To verify that the new boundary is indeed convex we use a $\theta$-invariant contact vector field $X$ whose restriction to $U(\Gamma)$ is of the form
\begin{align}
    X = a(t) \partial _ t + \dot{a}(t) \partial _ \tau,
\end{align}
and satisfies $\mathcal{L} _ {X} \alpha = \dot{a} \alpha$ where $a$ is a non-decreasing function that is equal to $t / C$ for $|t| \leq \frac{\epsilon}{2}$ and to $1/C$ (resp. $-1/C$) for $t \geq 1 - \frac{\epsilon}{2}$ (resp. $t \leq -1 + \frac{\epsilon}{2}$.) With $X$ so defined, it is straightforward to check that it is transverse to the boundary. 

Thus in a collar neighborhood of the boundary of the resulting contact manifold $(u(M) , u(\alpha) = \alpha | _ {u(M)})$ with convex boundary defined by the flow of $X$, $\alpha$ can be expressed as
\begin{align}
    \label{unsuturing_1-form}
    e ^ {\frac{s}{C}} \big( tds + C dt + e ^ {\Xi(t)} d \theta \big),
\end{align}
 where the flow of $X = \partial _ s$ can be solved explicitly as $(t , \tau , \theta) \mapsto (e ^ {\tfrac{s}{C}} t , \tau + \tfrac{s}{C} , \theta)$ with $|t|\leq\frac{\epsilon}{2}$. We shall refer to $(u(M) , u(\alpha) = \alpha | _ {u(M)})$ as the \emph{unsuturing} of $(M,\Gamma, U(\Gamma),\mathit{ker}(\alpha))$.

Next we outline the general recipe of switching from convex to sutured boundary condition following the proof of \cite[Lemma 4.1]{cghh}. Let $(M, \alpha)$ be a contact 3-manifold with convex boundary then in a collar neighborhood $(-\infty , 0] _ s \times \Sigma$ of $\partial M = \Sigma$ supplied by a contact vector field $X$, the contact form reads
\begin{align*}
    m (f ds + \beta)
\end{align*}
where $m$ is determined by the re-scaling factor of the Lie derivative of $\alpha$ with respect to $X$ and both $f$ and $\beta$ are $s$-invariant. Recall that $\Sigma_\pm$ are determined by $f > 0$ or $f < 0$, and the dividing set $\Gamma \subset \Sigma$ is cut out by the condition $f=0$. With the preceding understood, the suturing process can be summarized in six steps:

\begin{enumerate}\leftskip-0.25in
    \item 
    \label{step:(1)_suturing}
    By a positive conformal re-scaling we may assume that $f$ is only a function of $\tau$ within an annular neighborhood of the dividing set $A(\Gamma) \simeq [-1 ,1] _ \tau \times \Gamma _ \theta$ and that it equals $\pm 1$ on $\Sigma_\pm \cap \{ |\tau| \geq 1/4\}$. In order to guarantee that the restriction of $\alpha$ to $R _ \pm(\Gamma)$ away from a neighborhood of $\Gamma$ to be Liouville forms, we may require the associated Reeb vector field to be positively/negatively transverse along $R _ \pm(\Gamma)$, which can be achieved by re-scaling $\alpha$ so that $f = \pm 1$ on $R _ \pm(\Gamma)$. 
    \item 
    Using the characteristic line field directed by the vector field whose $\tau$-component is $\partial _ \tau$, we can parameterize $A(\Gamma)$ so that we have no $d \tau$-term in the restriction of $\beta $ to $A(\Gamma)$. 
    \item 
    \label{step:(3)_suturing}
    Since $\frac{d \beta}{d \tau}$ shares the same kernel as $\beta$, there is a positive function $a(\tau , \theta)$ so that $\frac{d \beta}{d \tau} = a(\tau , \theta) \beta$ and hence 
    \begin{align*}
        \beta (\tau , \theta) 
        & = \exp (\int _ {0} ^ {\tau} a(\zeta , \theta) d\zeta)\,\beta(0 , \theta) 
    \end{align*}
    Thus $\alpha$ can be written as
    \begin{align*}
        m(f(\tau) ds + g(\tau , \theta) \beta _ 0)
    \end{align*}
    where $\beta_0 = \beta(0,\theta)$ and $g$ is a positive function that satisfies
    \begin{itemize}\leftskip-0.25in
        \item $\frac{\partial g}{ \partial \tau} < 0 $ if $\tau >0$, and
        \item $\frac{\partial g}{ \partial \tau} > 0 $ if $\tau <0$.
    \end{itemize}
    \item 
    To further remove the $\theta$-dependence on $g$, first consider the following contact form 
    \begin{align*}
        \widetilde{\alpha} := m(f(\tau) ds + \widetilde{g} (\tau , \theta) \beta _ 0)
    \end{align*}
    where $\widetilde{g}$ is a positive function such that
    \begin{itemize}\leftskip-0.25in
        \item $\widetilde{g} = g$ when $\tau$ is near to $\pm 1$,
        \item $\widetilde{g} = \widetilde{g}(\tau)$ when $|\tau| < 1/2$,
        \item $\widetilde{g}$ satisfies the same conditions as $g$ listed above.
    \end{itemize}
    Then apply Moser-Weinstein technique to the linear isotopy starting with $\widetilde{\alpha}$ and ending with $\alpha$ to reparameterize $A(\Gamma)$ so that under the new parameterization $\alpha$ can be written as 
    \begin{align*}
        \alpha = \eta m (f(\tau) ds + \widetilde{g}(\tau) \beta _ 0),
    \end{align*}
    for $|\tau| < 1/2$, where $\eta$ is a positive function.
    \item 
    \label{step:(5)_suturing}
    Let $u = f(\tau) \widetilde{g}(\tau) ^ {-1}$ to write 
    \begin{align*}
    \alpha = \eta m \widetilde{g}^{-1} (u ds + \beta _ 0),
    \end{align*}
    or equivalently as
    \begin{align*}
    \alpha = h _ 0 (u ds + \beta _ 0).
    \end{align*}
    Note that by the contact condition $du/d\tau > 0$, hence $u$ is a new coordinate. Next use the following map from \cite{kevin_sackel}
    \begin{align*}
        \phi : (u , s , x) \mapsto (u ,s , \psi ^ {\frac{1}{2} us} (\theta)),
    \end{align*}
    where $\psi$ denotes the flow of the vector field on $\Gamma$ dual to $\beta _ 0$. The map $\phi$  satisfies
    \begin{align*}
    \phi ^ {*} (\tfrac{1}{2} (uds - sdu ) + \beta _ 0) = uds + \beta _ 0
    \end{align*}
    and with the polar coordinate on the $us$-plane we arrive at
    \begin{align*}
        \alpha = h _ 0 (\tfrac{1}{2} r ^ 2 d \varphi + \beta _ 0)
    \end{align*}
    defined on the half disk
    \begin{equation}
        \label{suturing_half_disk}
        U = \{ (u,s) = (r \cos \varphi , r \sin \varphi)\; |\; 0 \leq r \leq \delta , \pi \leq \varphi \leq 2 \pi \}.
    \end{equation}
    \item 
    \label{step:(6)_suturing}
    Finally, within $U$ perturb $h _ 0$ to get a function $h = h (r)$ that satisfies
    \begin{itemize}\leftskip-0.25in
        \item $\tfrac{\partial h}{\partial r} < 0$,
        \item $h = C_ 0 / r^{2}$ when $\varepsilon / 2 \leq r \leq \varepsilon$ and $\varepsilon$ is some positive constant $< \delta$.
    \end{itemize}
\end{enumerate}
Now we are ready to prove the following.
\begin{thm}
\label{thm:suturing_unsutured}
For each $L>0$, there exists a canonical isomorphism \[\ech^L(s(u(M)),\Gamma', s(u(\alpha)))\isomto \ech^L(M , \Gamma, \alpha).\]
\end{thm}
\begin{proof}
We start by arguing that no new closed Reeb orbits are introduced as a result of suturing the unsutured. To see this, note that equation \eqref{unsuturing_1-form} suggests that Step \eqref{step:(1)_suturing} of the suturing process can be omitted and all potential changes in Reeb dynamics lie within the half disk $U$ specified in \eqref{suturing_half_disk}. The fact that the restriction of the Reeb vector field for $s(u(\alpha))$ to $U\times\Gamma'$ is positively transverse to every $\varphi$-constant hyperplane, which is guaranteed by our choice of the function $h$ in Step \eqref{step:(6)_suturing} of the suturing process (cf. proof of \cite[Lemma 4.1]{cghh}), implies that no integral curve of the Reeb vector field intersecting $U(\Gamma) \cap s(u(M))$ would close up. In fact, every maximal integral curve of the Reeb vector field for $s(u(\alpha))$ starts at a point on $R_-(\Gamma')$ and ends at a point on $R_+(\Gamma')$. Consequently, closed Reeb orbits of $(s(u(M)), s(u(\alpha)))$ lie in the interior of $s(u(M))\setminus U(\Gamma)$, hence they are the same as closed Reeb orbits of $(M,\alpha)$. 

Next, we use intersection theory for holomorphic curves as in the proof of \cite[Lemma 3.1]{kst} to argue that no holomorphic curve that contributes to the differential on $\ech(s(u(M)),\Gamma', s(u(\alpha)))$ projects onto $U (\Gamma)\cap s(u(M))$. Suppose that such a curve $\mathcal{C}$ exists and fix a point $p_0$ in $U (\Gamma)\cap s(u(M))$ that lies in the projection of $\mathcal{C}$. Let $\gamma_0$ denote the maximal integral curve of the Reeb vector field for $s(u(\alpha))^\ast$ through $p_0$ and denote by $x_0$ its intersection with $R_+(\Gamma')$. Next, let $x_1\in \widehat{R _ + (\Gamma ' )} \setminus R _ + (\Gamma ')$ where $\widehat{R _ + (\Gamma ')}$ denotes the completion of the Liouville domain $R _ + (\Gamma')$ and $\gamma_1$ be the maximal integral curve of the Reeb vector field for $s(u(\alpha))^\ast$ through $x_1$. Since $\mathcal{C}$ is asymptotic to Reeb orbits contained in the interior of $M\setminus U(\Gamma)\subset s(u(M))$, there exit sufficiently large $S,T>0$ such that the $s-$ and the $t-$coordinates of $\mathcal{C}$ in $\mathbb{R}_s\times (s(u(M))^\ast\setminus s(u(M)))$ have absolute values bounded respectively by $S$ and $T$. Having fixed a path $P:[0,1]\to \widehat{R _ + (\Gamma ')}$ connecting $x_0$ to $x_1$ in the complement of $\overline{R_+(\Gamma)\setminus U(\Gamma)}$ in $R_+(\Gamma')$, let $\widetilde{P}:D_P\to s(u(M))^\ast$ be its lift to a 2-disk $D_P$ by integral curves of the Reeb vector field for $s(u(\alpha))^\ast$ whose $t-$coordinates are constrained to satisfy $|t|\leq T$, and consider the $3$-chain
\begin{align*}
    A=id \times \widetilde{P} : [-S , S] \times D_P \rightarrow \mathbb{R} _ s \times s(u(M)) ^ {*}.
\end{align*}
\noindent Note that $\partial A$ is null-homologous, hence $\mathcal{C}\cdot\partial A=0$. Since the $s-$ and the $t-$ coordinates of $\mathcal{C}\cap\partial A$ cannot be equal to $\pm S$ and $\pm T$, respectively, we conclude that 
\[0=\mathcal{C}\cdot\partial A=\mathcal{C}\cdot(\mathbb{R}\times\gamma_1)-\mathcal{C}\cdot(\mathbb{R}\times\gamma_0).\]
Moreover, $\mathcal{C}\cdot(\mathbb{R}\times\gamma_1)=0$ because $\mathcal{C}$ does not project onto $\widehat{R _ + (\Gamma ' )} \setminus R _ + (\Gamma ')$ by \cite[Lemma 5.5]{cghh}. As a result, $\mathcal{C}\cdot(\mathbb{R}\times\gamma_0)=0$, contradicting our assumption.
\end{proof}
\subsection{Contact 2-handle attachment in sutured ECH}
\label{ss:contact_2-handle}
We now describe how to attach a contact 2-handle to a sutured contact 3-manifold $(M,\Gamma,U(\Gamma),\mathit{ker}(\alpha))$ along an attaching curve $\gamma$ intersecting the suture $\Gamma$ at exactly two points. This shall be done in three steps:
\begin{enumerate}\leftskip-0.25in
    \item Unsuture the sutured contact 3-manifold $(M,\Gamma,U(\Gamma),\mathit{ker}(\alpha))$ to get a contact 3-manifold $(u(M),u(\alpha))$ with convex boundary where $\Gamma$ becomes the dividing set.
    \item Attach a contact 2-handle on $(u(M),u(\alpha))$ following Giroux's recipe \cite{giroux-convex} along an isotopic copy of $\gamma$ which agrees with $\gamma$ in the complement of where rounding of corners occurs. 
    \item Suture the resulting contact 3-manifold $(u(M)_2,u(\alpha)_2)$.
\end{enumerate}
For the first step, we further require that the contact vector field as in the beginning of Section \ref{ss:sutured-convex} is chosen so that there is a collar neighborhood of the dividing set $A(\Gamma) \simeq \Gamma \times [-1,1]$ on $\Sigma=\partial u(M)$ with which 
\begin{itemize}\leftskip-0.25in
    \item 
    the characteristic foliation of $\ker \alpha | _ {u(M)}$ on $A(\Gamma)$ is an interval-fibration over $\gamma$ parametrized as $ \gamma \times [-1,1]$ so that $\partial A(\gamma) = \gamma \times \{ \pm 1\}$, and
    \item 
    the function $f$ in equation \eqref{eq:convex_form} is equal to $\pm 1$ over $\Sigma_\pm \setminus (\Gamma \times [-1/2 , 1/2])$, where the restriction of $\beta$ defines Liouville forms.
\end{itemize}

For the second step, we first match the characteristic foliations on a neighborhood $A(\gamma)$ of the attaching curve and that of the attaching circle as in Giroux's contact 2-handle model 
$(H _ {2} , \zeta _ 2 = \ker \alpha _ 2)$ where $H _ 2 = \{ (x,y,z) \in \mathbb{R} ^ 3\; |\; x ^ 2 + z ^ 2 \leq 1 , |y| \leq \varepsilon\}$ and $\alpha _ 2 = dz + ydx + 2xdy$ (see \cite[Section III.3.B]{giroux-handle} and \cite[Section 1]{Ozbagci_bypass}). Then we match the contact structures. To elaborate, construct a 1-form $\beta _ 0$ on $\partial u(M)$ satisfying
\begin{itemize}\leftskip-0.25in
    \item $\beta _ 0 = \beta$ on $A(\Gamma)$,
    \item $\beta _ 0$ is a Liouville form on $\Sigma _ {\pm} \setminus (\Gamma \times [-1/2 , 1/2])$,
    \item the kernel of $\beta _ 0$ defines a singular foliation on $\partial u(M)$ whose restriction to $A(\gamma)$ coincides with that around the attaching circle of $(H _ 2 , \zeta _ 2)$.
\end{itemize}
Then apply the Moser-Weinstein technique to the family of contact forms 
$\alpha _ {\nu} = m(fds + \beta _ {\nu})$ where $\beta _ {\nu} = \beta _ {0} (1 - \nu) + \beta \nu$ so as to solve the equation
\begin{align*}
    \phi _ {\nu} ^{*} \alpha^{} _ {\nu} = \alpha _ 0
\end{align*}
as in \cite{kst}. Use the time-1 map $\phi _ 1$ to isotope $u(M)$ within a collar neighborhood so that $\ker \alpha _ {1} = \ker \alpha$ induces the required characteristic foliation over $\phi _ 1 (A(\gamma))$. Note that
\begin{itemize}\leftskip-0.25in
    \item 
    the isotopy $\phi _ {\nu}$ is the identity over the region $A(\Gamma) \times \mathbb{R}_{s}$ where $\alpha _ {\nu}$ is constant,
    \item 
    the perturbed boundary $\phi _ 1 (\partial u(M))$ is still transverse to $\partial _ s$, and
    \item 
    $\phi _ \nu$ is a \emph{strict} contactomorphism.
\end{itemize}
Consequently, we have
\begin{corollary}
    \label{cor:s(u(M))=M}
    For each $L>0$, there exists a canonical isomorphism \[\ech^L(s(\widetilde{u}(M)),\Gamma', s(\widetilde{\alpha}))\isomto \ech^L(M , \Gamma, \alpha)\]
    where $\widetilde{u}(M)$ is the manifold obtained from $u(M)$ after a perturbation needed to match the characteristic foliation of $A(\gamma)$ with that of $(H _ 2 , \zeta _ 2)$ around its attaching circle; and $\widetilde{\alpha}$ denotes the restriction of $\alpha$ to $\widetilde{u}(M)$.
\end{corollary}
\begin{proof}
This follows readily from Theorem \ref{thm:suturing_unsutured} and the fact that there is a strict contactomorphism between $(\widetilde{u}(M),\widetilde{\alpha})$ and $(u(M),u(\alpha))$ fixing a collar neighborhood of $A(\Gamma)$.
\end{proof}
To complete the description of the contact 2-handle attachment process, we note that if $\psi : F _ 2 \rightarrow \phi _ 1 (A(\gamma))$ is an attaching map that intertwines the characteristic foliations then it follows from \cite[III.3.B]{giroux-convex} that there is an isotopy $\varphi _ {\nu}$ of $\phi _ 1 (\partial u(M))$ and a positive function $\mu$ such that 
\begin{align*}
\widetilde{\alpha} = ((\varphi _ 1 \circ \psi )^{-1})^ {*}(\mu \alpha _ 2)    
\end{align*}
In other words, the contact manifold $(H _ 2 , \mu \alpha _ 2)$ can be attached along the curve $\varphi _ 1 (\psi (\phi _ 1(\gamma)))$ and extends the contact form $\widetilde{\alpha}$ on $\widetilde{u}(M)$. We denote the resulting contact manifold with convex boundary by 
$(\widetilde{u}(M) _ 2 , \widetilde{\alpha} _ 2)$.

\section{Contact 2-handle attachments and contact surgery}
\label{sec:contact_2-handles}
In this section, we adapt Baldwin and Sivek's recipe in \cite{bs2} to define contact 2-handle attachment maps on sutured monopole homology to sutured ECH. Baldwin and Sivek show in \cite[Section 4.2.3]{bs2} that attaching a contact 2-handle on a contact 3-manifold with convex boundary along a curve $\gamma$ that intersects the dividing set in exactly two points followed by an auxiliary contact 1-handle attached along the belt sphere of the 2-handle is contactomorphic to contact $(+1)$-surgery along a parallel copy $\gamma'$ of $\gamma$ in the interior of the manifold. This is because $\gamma'$ can be chosen to be a Legendrian knot by the Legendrian Realization Principle, and since $\gamma$ intersects the dividing set in exactly two points, contact $(+1)$-surgery along $\gamma'$ is the same as boundary-framed surgery. The following proposition describes Baldwin and Sivek's contactomorphism in our context. We include its proof for sake of completeness and for future reference.
\begin{prop}[cf. Section 4.2.3 in \cite{bs2}] Let 
    \label{prop:BS}
    $\widetilde{u}(M)_{2,1}$ denote the contact 3-manifold with convex boundary obtained from $\widetilde{u}(M)$ by attaching a contact 2-handle along a simple closed curve $\gamma\subset \partial\widetilde{u}(M)$ intersecting the dividing set of $\widetilde{u}(M)$ in exactly two points followed by an auxiliary contact 1-handle attached along the belt sphere of that 2-handle, $\widetilde{u}(M)'$ denote the result of performing contact (+1)-surgery on $\widetilde{u}(M)$ along $\gamma'$, and $\widetilde{\alpha}_{2,1}$  denote the extension of $\widetilde{\alpha}$ to $\widetilde{u}(M)_{2,1}$. Then there is a contactomorphism 
    \[
    \mathbf{BS} : (\widetilde{u}(M)_{2,1} , \widetilde{\alpha}_{2,1}) \rightarrow (\widetilde{u}(M)',  \widetilde{\alpha}'),
    \]
    where $\widetilde{\alpha}'= (\mathbf{BS}^{-1})^\ast\widetilde{\alpha}_{2,1}$. The contactomorphism $\mathbf{BS}^{-1}$ restricts to the identity in the complement of a vertically invariant neighborhood of $\gamma$.\qed
\end{prop}
\begin{proof}
    Fix a transverse contact vector field $X$ near $\partial\widetilde{u}(M)$ so as to define a vertically invariant collar neighborhood of $\partial\widetilde{u}(M)$. Following \cite[Section 4.2.3]{bs2}, use $X$ to define a vertically invariant neighborhood  $N$ of $\gamma$ and denote by $N_1$ the union of $N$ with the 2- and 1-handle pair attached as in \cite[Figures 18 and 19]{bs2}\footnote{It should be noted that in contrast with \cite{bs2}, contact handles in our case do not have Legendrian corner. Instead, we round corners after attaching contact handles as in \cite[Figure 5]{Ozbagci_bypass}.}. Round corners of $N$ so as to make its boundary a convex torus with two meridional dividing curves. Do this in such a way that $A(\gamma)=\partial N \cap \partial \widetilde{u}(M) \cong [-1,1] \times \gamma$.

    The identity map 
    \[
    id : \widetilde{u}(M)_{2,1} \setminus N _ 1 \rightarrow \widetilde{u}(M) \setminus N
    \]
    extends to a contactomorphism between some collar neighborhoods of $\partial N _ 1$ and $\partial N$, which is unique up to isotopy. To be more precise, note that part of the boundary of $N _ 1$ agrees with that of $N$ and replaces $A(\gamma)$ by another annulus which we denote by $A _ 1$. We may assume that the two intersect near their common boundary 
    \[
    A(\gamma) \cap A _ 1 \cong ([-1 , -\frac{3}{4}] \cup [\frac{3}{4} , 1]) \times \gamma.
    \]
    Next fix a diffeomorphism $\sigma : A _ 1 \rightarrow A (\gamma)$ that restricts to the identity on $A(\gamma) \cap A _ 1$ and use $\sigma$ to extend $id : \partial N _ 1 \setminus A _ 1 \rightarrow \partial N \setminus A (\gamma)$ to a diffeomorphism $\overline{\sigma} : \partial N _ 1 \rightarrow \partial N$ that identifies the dividing sets. We may assume that $\sigma$ intertwines the characteristic foliations on $A_1$ and $A(\gamma)$. Otherwise, perturb $\partial N_1$ via an isotopy supported in a vertically invariant neighborhood of $A _ 1 \setminus A(\gamma)$ and replace $\sigma$ by $\rho _ 1 \circ \sigma$ where $\rho _ t : A(\gamma) \rightarrow A (\gamma)$ is an isotopy that is identity on $([-1 , -\frac{3}{4}] \cup [\frac{3}{4} , 1]) \times \gamma$. The extension of $\overline{\sigma}$ to some vertically invariant collar neighborhoods of $\partial N _ 1$ and $\partial N$ then extends $id : \widetilde{u}(M)_{2,1} \setminus N _ 1 \rightarrow \widetilde{u}(M) \setminus N$ to a contactomorphism
    \[
    f:
    \widetilde{u}(M)_{2,1} \setminus {N _ 1}'
    \rightarrow 
    \widetilde{u}(M) \setminus N', 
    \]
    where ${N_1}'\subset\textrm{int}(N_1)$ and $N'\subset \textrm{int}(N)$ are the complements of the aforementioned collar neighborhoods of $\partial N_1$ and $\partial N$, respectively. 
    Finally, glue the solid torus ${N _ 1}'$ onto $\widetilde{u}(M) \setminus N'$ in a way that extends $f$. The map $\mathbf{BS}$ is nothing but the extension of $f$ over ${N _ 1}'$. Note that gluing ${N_1}'$ onto $N\setminus N'$ results in a solid torus obtained by performing boundary-framed surgery along $\gamma'$ in $N$. In fact, the restriction of $\mathbf{BS}$ to $N_1$ is a diffeomorphism between two solid tori that identifies the collar neighborhoods of boundary tori and interchanges the meridian and the longitude. Since $\gamma$ intersects the dividing set on $\partial N$ in exactly two points, the resulting contact manifold $\widetilde{u}(M)'$ is contactomorphic to contact (+1)-surgery on $(\widetilde{u}(M),\textrm{ker}(\widetilde{\alpha}))$ along~$\gamma'$.
\end{proof}

With the above proposition understood, it suffices to show that contact (+1)-surgery along $\gamma'$ induces a homomorphism between sutured $\ech$ of the two sutured contact 3-manifolds before and after the surgery.
\begin{prop}
    \label{prop:handle_surgery}
    Let $(M, \alpha)$ be a sutured contact 3-manifold and $K$ be a Legendrian knot in the interior of $M$. Then there exists a well-defined map 
    \[\mathbf{F} _ {K}: \ech(M , \Gamma , \alpha) \rightarrow \ech(M ' , \Gamma , \alpha_{K})\]
    where $M'$ is the result of contact $(+1)$-surgery along $K$, and $\alpha'$ is a certain contact form extending the restriction of $\alpha$ to the complement of a standard neighborhood of $K$.  
\end{prop}
\begin{proof}
 As is explained in \cite[Section 2.3]{kst}, given $L>0$ there exists a closed, contact 3-manifold $(Y_n,\alpha_n)$ such that $(M,\alpha)$ embeds into and both have the same set of closed Reeb orbits of symplectic action less than $L$ (see also \cite[Section 10.3]{cgh-openbook}). Here $n$ is a positive number that can be chosen to be arbitrarily large as long as it is sufficiently large with respect to $L$. Furthermore, $\ech^L(M , \Gamma , \alpha)$ and $\ech^L(Y_n , \alpha_n)$ are canonically isomorphic by \cite[Lemma 3.5]{kst}. By \cite[Proposition 6.4.3]{geiges_intro}, attaching a symplectic 2-handle along $K$ in the concave end of the symplectic cobordism $([0,1] _ s \times Y_n , d(e^{s} \alpha_n) )$ results in an exact symplectic cobordism $(X_n,\omega_n)$. Then the map $\mathbf{F}_K$ arises as the direct limit of maps obtained by pre-composing the maps
\[{F_K}^L : \ech^{L}(M , \Gamma , e\alpha)
\rightarrow 
\ech^L(M ' , \Gamma , \alpha_{K})\]
induced by $(X_n,\omega_n)$ with the canonical scaling isomorphism \cite[Theorem 1.3(d)]{ht2}
\[
\ech (M , \Gamma , \alpha) 
\isomto
\ech (M, \Gamma , e \alpha).
\]
Details of the construction of these maps and their properties are deferred to Section \ref{sec:proof33}.
\end{proof}

Composing the maps we have so far, we arrive at a prototypical version of the desired 2-handle attachment map:
\begin{thm}
    \label{thm:pre_2_handle_map}
    There exists a well-defined map 
    \[\underline{F}_{h_2}:\ech(M,\Gamma,\alpha)\to\ech(s(\widetilde{u}(M)_{2,1}),\Gamma'',s(\widetilde{\alpha}_{2,1})).\]
\end{thm}

\begin{proof} The map $\underline{F}_{h_2}$ arises as the composition of the maps 
\begin{eqnarray}
    \label{eq:prop26}\ech(M,\Gamma,\alpha)&\xrightarrow{\simeq}&\ech(s(\widetilde{u}(M)),\Gamma',s(\widetilde{\alpha}))\\
    \label{eq:prop32}\ech(s(\widetilde{u}(M)),\Gamma',s(\widetilde{\alpha}))&\xrightarrow{\mathbf{F}_{\gamma'}}&\ech(s(\widetilde{u}(M)'),\Gamma',s(\widetilde{\alpha}_{\gamma'}))\\
    \label{eq:confscaling}\ech(s(\widetilde{u}(M)'),\Gamma',s(\widetilde{\alpha}_{\gamma'}))&\xrightarrow{\simeq}&\ech (s(\widetilde{u}(M)'),\Gamma',s(\widetilde{\alpha}'))\\
    \label{eq:prop31}\ech (s(\widetilde{u}(M)'),\Gamma',s(\widetilde{\alpha}'))&\xrightarrow[\mathbf{BS}^{-1}]{\simeq}&\ech(s(\widetilde{u}(M)_{2,1}),\Gamma'',s(\widetilde{\alpha}_{2,1}))
\end{eqnarray}
    The isomorphism in \eqref{eq:prop26} is supplied by Proposition \ref{cor:s(u(M))=M}, while the isomorphism in \eqref{eq:prop31} is provided by Proposition \ref{prop:BS}. Note that each of these isomorphisms are induced by chain maps that send the empty orbit set to the empty orbit set. The map $\mathbf{F}_{\gamma'}$ in \eqref{eq:prop32} is supplied by Proposition \ref{prop:handle_surgery}, and by \cite[Theorem 1.9(ii)]{ht0} along with the subsequent Remark 1.11 it sends the ECH contact class to the ECH contact class. As for the isomorphism in \eqref{eq:confscaling}, it follows from the proofs of Propositions \ref{prop:BS} and \ref{prop:handle_surgery} that $\widetilde{\alpha}_{\gamma'}=\widetilde{\alpha}'$ in the complement of a standard neighborhood of $\gamma'$ in the interior of $\widetilde{u}(M)$ and that
    $(\widetilde{u}(M)',\textrm{ker}(\widetilde{\alpha}_{\gamma'}))$ and $(\widetilde{u}(M)',\textrm{ker}(\widetilde{\alpha}'))$ are contactomorphic via a contactomorphism that is equal to the identity in the complement of that neighborhood. By requiring that the suturing operation as described in Section \ref{ss:sutured-convex} take place in the complement of the standard neighborhood of $\gamma'$, we guarantee that the sutured contact 3-manifolds $(s(\widetilde{u}(M)'),\Gamma',s(\widetilde{\alpha}_{\gamma'}))$ and $(s(\widetilde{u}(M)'),\Gamma',s(\widetilde{\alpha}'))$ are contactomorphic via a contactomorphism that is equal to the identity in the complement of the standard neighborhood of $\gamma'$, which contains a collar neighborhood of $\partial s(\widetilde{u}(M)')$. The isomorphism in \eqref{eq:confscaling} is induced by this contactomorphism, which of course maps the empty orbit set to the empty set.
\end{proof}
In the next section, we show that there exists an isomorphism
\begin{equation}
\label{eqn:suturevs1handle}
\ech\big(s(\widetilde{u}(M)_{2,1}),\Gamma'',s(\widetilde{\alpha}_{2,1}) \big)
\xrightarrow{\sim}
\ech\big( s(\widetilde{u}(M) _ 2),\Gamma', s(\widetilde{\alpha} _ 2)\big)
\end{equation}
which when composed with the map in Theorem \ref{thm:pre_2_handle_map} gives the desired contact 2-handle attachment map in Theorem \ref{thm:main}.

\section{Suturing and contact 1-handle attachment commute}
\label{sec:suturingvshandle}
By \cite[Lemma 3.1]{kst}, $\ech(s(\widetilde{u}(M) _ 2) , \Gamma', s (\widetilde{\alpha} _ 2))$ is canonically isomorphic to $\ech(s(\widetilde{u}(M) _ 2) _ 1,\Gamma', s (\widetilde{\alpha} _ 2) _ 1)$ where $(s(\widetilde{u}(M) _ 2) _ 1,\Gamma', s (\widetilde{\alpha} _ 2) _ 1)$ is the sutured contact 3-manifold obtained from $(s(\widetilde{u}(M) _ 2) , \Gamma', s (\widetilde{\alpha} _ 2))$ by attaching a contact 1-handle as described in \cite[Section 2.2]{kst}. Thus, to show the existence of the isomorphism \eqref{eqn:suturevs1handle} it suffices to prove the following.

\begin{lemma}
There exists an isomorphism 
\[\ech \big( s(\widetilde{u}(M)_{2,1}) , s(\widetilde{\alpha }_{2,1}) \big) \xrightarrow{\sim} 
\ech \big( s(\widetilde{u}(M)_2)_1 , s(\widetilde{\alpha}_2)_1 \big),\] i.e the suturing operation commutes with the contact 1-handle attachment.
\end{lemma}

\begin{proof}
Our proof has three steps.\\
\noindent\textit{Step 1.} Note that the annulus $A(\gamma) \subset \widetilde{u}(M)$ in the proof of Proposition \ref{prop:BS} becomes a disjoint union of two disks after attaching the contact 2-handle. We may assume that 
\begin{figure}[ht]
\includegraphics[width=12cm]{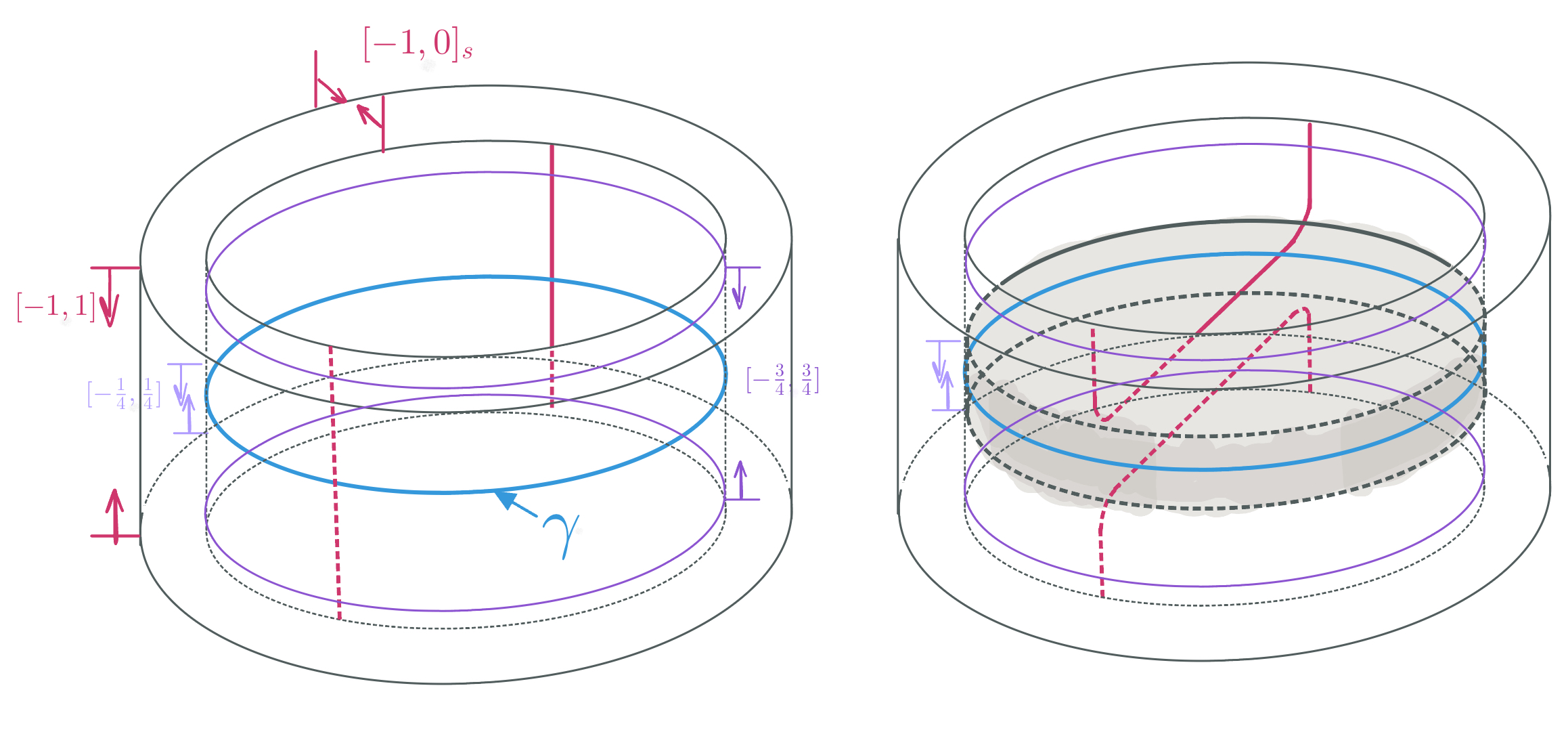}
\caption{Left: $[-1 , 0] _ s \times A(\gamma)  \simeq [-1 , 0] _ s \times [-1 , 1] \times \gamma $ in which the red arcs denote the part of the dividing curve on $\widetilde{u}(M)$. Right: After the $2$-handle is attached along $[-\tfrac{1}{4} , \tfrac{1}{4}] \times \gamma$ the part of $\partial \widetilde{u}(M) _ 2$ in this local picture topologically is a disjoint union of two disks which we call $D$ and $D'$ refers to the complement of 
$\{ 0 \} _ s \times ([ -1 , -\tfrac{3}{4}) \cup ( \tfrac{3}{4} , 1] ) \times \gamma$ in $D$.
} 
\label{fig:D_prime_and_D}
\centering
\end{figure}    \begin{itemize}\leftskip-0.25in
        \item the 2-handle is attached along the part of $A (\gamma)$ parametrized by 
        \[
        [-\frac{1}{4} , \frac{1}{4}] \times \gamma,
        \]
        \item the isotopy supported in a neighborhood of $A _ 1 \setminus A (\gamma)$ used in the construction of the map $\mathbf{BS}$ from Proposition \ref{prop:BS} is the identity on 
        \[
        \overline{\widetilde{u}(M) _ 2 \setminus [-\frac{1}{2} , 0] _ s \times D '}
        \]
        where $D' \subset \partial \widetilde{u}(M) _ 2$ denotes the disjoint union of the two disks whose boundaries are identified with $\{ 0 \} _ s \times \{ \pm \frac{3}{4}\} \times \gamma$, see Figure \ref{fig:D_prime_and_D}.
    \end{itemize}
Let $D$ denote an analogous version of the above $D'$ with $\{ \pm \frac{3}{4}\}$ replaced by $\{ \pm 1 \}$ then 
\[
s (\widetilde{u}(M) _ {2,1} \setminus \big( [-1 , 0] _ s \times D \cup h _ 1 \big)) 
=
s (\widetilde{u}(M) _ 2) _ 1 \setminus \big( [-1 , 0 ] _ s \times D \cup H _ 1 \big)
\]
and the identity map is a strict contactomorphism in which 
\begin{itemize}\leftskip-0.25in
    \item $h _ 1$ denotes the auxiliary 1-handle attached on $\widetilde{u}(M)_2$ with the contact form that is the restriction of $\widetilde{\alpha}_{2,1}$ to $h_1$,
    \item $H _ 1$ denotes the auxiliary 1-handle attached on $s(\widetilde{u}(M)_2)$ with the contact form that is the restriction of $s(\widetilde{\alpha}_2)_1$ to $H_1$.
\end{itemize}
To see this, it suffices to note that the two suturing operations, one for $\widetilde{u}(M)  _ {2,1}$ and the other for $\widetilde{u}(M)_2$, produce the same complement in $\widetilde{u}(M) _ 2 \setminus \big( [-1 , 0] _ s \times D \big) $, the latter being contained in both $\widetilde{u}(M) _ {2,1}$ and $\widetilde{u}(M) _ 2$. Indeed, $\widetilde{u}(M) _ 2 \setminus \big( [-1 , 0] _ s \times D \big) $ is contained in the open neighborhood 
$\widetilde{u}(M) _ 2 \setminus \big( [-\frac{1}{2} , 0] _ s \times D '\big)$ over which $\widetilde{\alpha} _ {2,1}$ and $\widetilde{\alpha} _ 2$ coincide.\\  

\noindent\textit{Step 2.} Next, we extend the identity map in the previous step to collar neighborhoods of the boundaries of $s(\widetilde{u}(M)_{2,1})$ and $s (\widetilde{u}(M)_2)_1$  by first fixing diffeomorphisms between the boundaries
\begin{align*}
\varphi ^ {\pm} &:  {\partial _ h }^ \pm s(\widetilde{u}(M) _ {2,1}) \rightarrow {\partial _ h }^ \pm s (\widetilde{u}(M) _ 2) _ 1\\
\varphi ^ {v} &:  \partial _ {v} s(\widetilde{u}(M) _ {2,1}) \rightarrow \partial _ {v} s (\widetilde{u}(M) _ 2) _ 1
\end{align*}
and then using the sutured contact manifold structure, \eqref{alpha_over_vert} and \eqref{alpha_over_pm}. The two suturing operations performed on 
$(\widetilde{u}(M) _ 2 , \widetilde{\alpha} _ 2)$ and $(\widetilde{u}(M) _ {2,1} , \widetilde{\alpha} _ {2,1})$ would produce an identical strict contact submanifold $E$ as depicted in Figure \ref{fig:chunk_E} and is embedded in the interior of 
$[-1 , 0] _ s \times ([-1 , -\frac{3}{4}] _ z \cup [\frac{3}{4} , 1] _ z) \times \gamma$ when the latter is regarded as both contained in $(\widetilde{u}(M) _ 2 , \widetilde{\alpha} _ 2)$ and $(\widetilde{u}(M) _ {2,1} , \widetilde{\alpha} _ {2,1})$. 

To be more precise, we describe $E$ as follows. First denote the portion of the dividing curve whose $z$-coordinate is between $-1$ to $-\frac{3}{4} - \kappa$ or $\frac{3}{4} + \kappa$ to $1$ where $0 < \kappa < \frac{1}{100}$ and are referred to as $I = I _ 0 \cup I _ 1$ and $J = J _ 0 \cup J _ 1$ respectively, see Figure \ref{fig:I_and_J}. Notice that both $I$ and $J$ have two components; each corresponds to one of the two intersections of $\gamma$ and $\Gamma$. Then use the flow of $\partial _ u$ as in the Step \eqref{step:(5)_suturing} of suturing to sweep $I$ and $J$ within a suitable amount of time say $\delta > 0$ and whose image is exactly $U \cap \{ s =0 \}$ where $U$ is introduced in Step \eqref{suturing_half_disk}. The time-$\pm \frac{\varepsilon}{2}$ images of both $I$ and $J$, denoted by $I ^ + = I _ 0 ^ + \cup I _ 1 ^ +$, $J ^ + = J _ 0 ^ + \cup J _ 1 ^ +$ and $I ^ - = I _ 0 ^ - \cup I _ 1 ^ -$, $J ^ -= J _ 0 ^ - \cup J _ 1 ^ -$ respectively would be part of the boundary circles of the plus/minus horizontal boundaries 
    ${\partial _ h }^ \pm s(\widetilde{u}(M) _ {2,1})$.

The entire $E$ can then be constructed by taking the collar neighborhood of the part of it that belongs to $\partial s(\widetilde{u}(M) _ {2,1})$ 
using the suture manifold structure \eqref{alpha_over_vert} and \eqref{alpha_over_pm}. 
Notice that as depicted in the Figure \ref{fig:chunk_E}, $E$ has two components each of which corresponds to $I$ or $J$. The part for $I$ alluded to consists of 
(i) $r = \frac{\varepsilon}{2}$ level set within $U \cap (I \times [-\delta , \delta] _ u \times [-1 , 0]_s)$, (ii) a rectangle embedded in ${\partial _ h} ^ + s(\widetilde{u}(M) _ {2,1}) \setminus [-\tfrac{3}{4} , \tfrac{3}{4}] \times \gamma $ with $I _ 0 ^ +$ and $I _ 1 ^ +$ as its opposite sides and another rectangle embedded in ${\partial _ h} ^ - s(\widetilde{u}(M) _ {2,1}) \setminus [-\tfrac{3}{4} , \tfrac{3}{4}] \times \gamma $ with $I _ 0 ^ -$ and $I _ 1 ^ -$ as its opposite sides see the red-shaded regions at the upper level in Figure \ref{fig:I_and_J}. Notice an analogous statement holds for $J$-part also and which is illustrated analogously as the red-shaded regions at the bottom level.  
    
Finally, the above overall construction works for $\partial s (\widetilde{u}(M) _ 2) _ 1$ and produces the same strict contact submanifold $E$ since all the constructions are staying away from the region where the auxiliary 1-handle is attached in either fashion: before or after suturing.
\begin{figure}[ht]
\includegraphics[width=5cm]{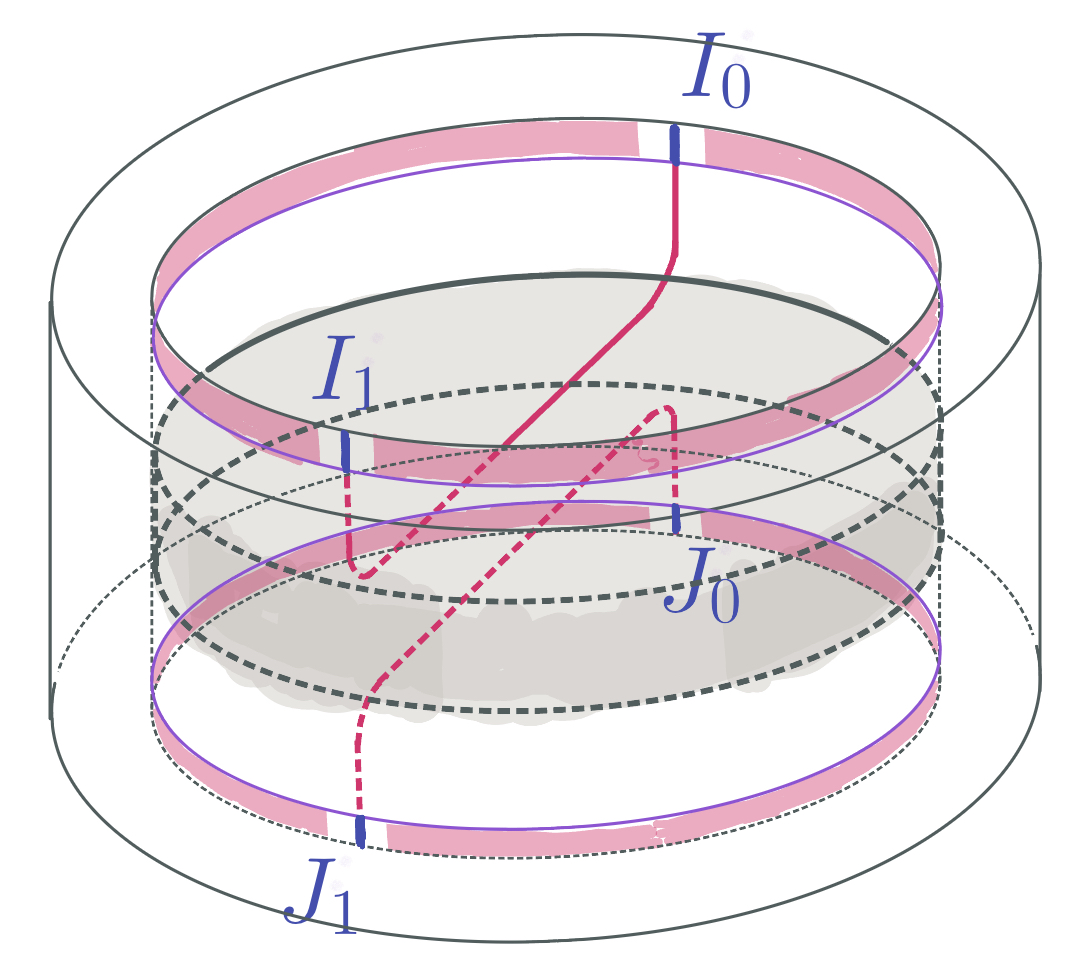}
\caption{}
\label{fig:I_and_J}
\centering
\end{figure}

Denote $\partial E\,\cap\, {\partial _ h }^ \pm s (\widetilde{u}(M) _ {2,1})$ by ${\partial _ h} ^ \pm E$ and $\partial E\,\cap\,\partial _ v s(\widetilde{u}(M) _ {2,1})$ by $\partial _ v E$. Then there are collections of properly embedded arcs $a _ {\pm} \subset {\partial _ h}^{\pm} E$ and $a _ v \subset \partial _ v E$ as portrayed by the purple arcs in Figure \ref{fig:chunk_E} that partition ${\partial _ h }^ \pm s(\widetilde{u}(M) _ {2,1})$ and $\partial _ v s(\widetilde{u}(M) _ {2,1})$ (resp. those for $s(\widetilde{u}(M) _ 2) _ 1$) into subsurfaces. We extend $\varphi ^ { \pm}$ and $\varphi ^ {v}$ so as to 
\begin{itemize}\leftskip-0.25in
    \item be equal to the identity over the subsurfaces that contain 
    ${\partial _ h }^ \pm s (\widetilde{u}(M) _ {2,1}) \setminus (D \times [-\delta , 0] _ s \cup h _ 1)$ or 
    $\partial _ v s(\widetilde{u}(M) _ {2,1}) \setminus (D \times [-\delta , 0] _ s \cup h _ 1)$, and
    
    \item satisfy $( \varphi ^ {\pm} ) _ {*} \partial _ {\tau} = \partial _ \tau $ and $(\varphi ^ {v}) _ {*} \partial _ t = \partial _ t$.
\end{itemize}
\begin{figure}[ht]
\includegraphics[width=5cm]{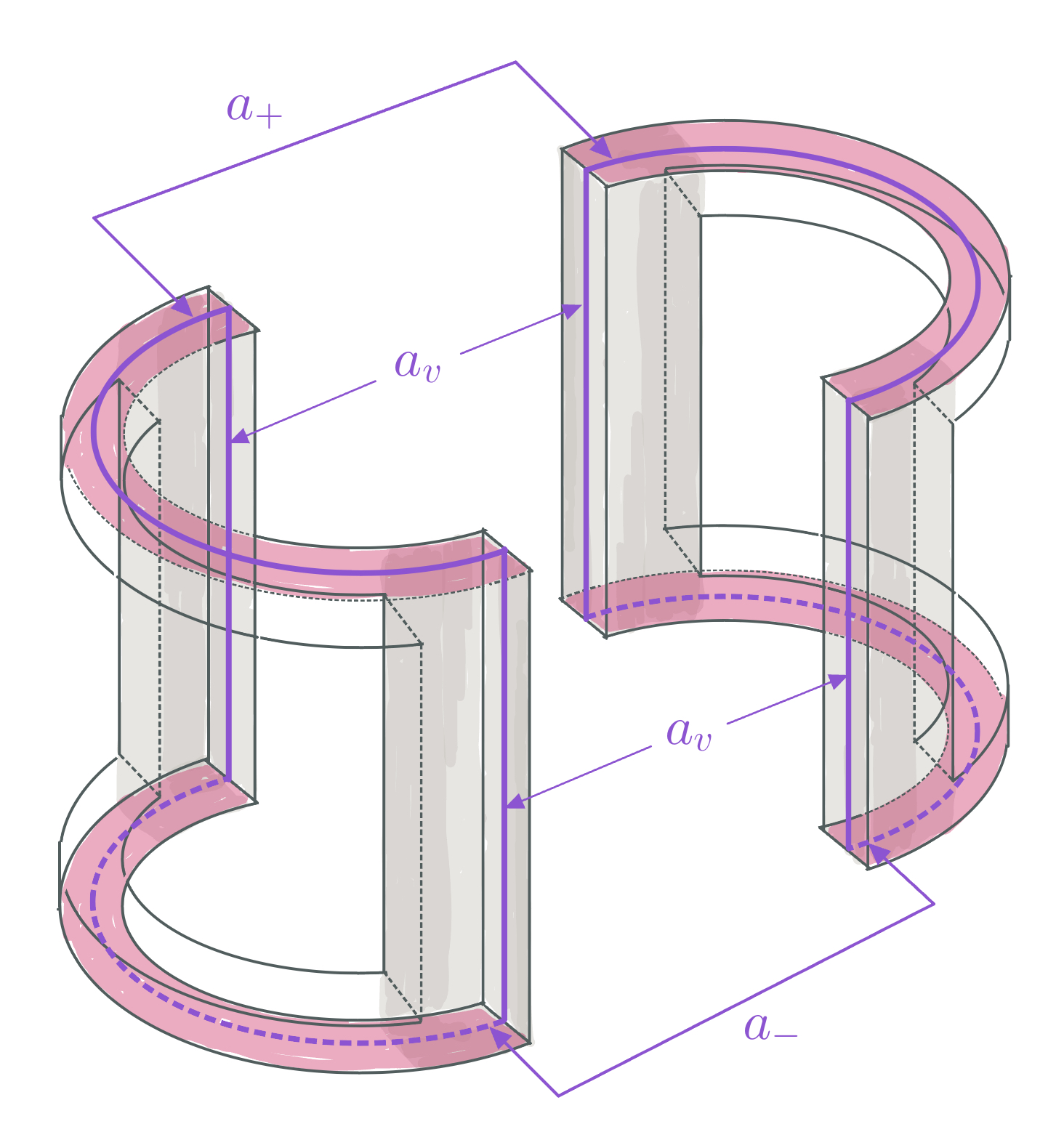}
\caption{The upper red-shaded region denotes ${\partial _ h} ^ + E$ while the bottom red-shaded region denotes ${\partial _ h }^ - E$ upon which lie $a _ +$ and $a _ -$, respectively. The gray-shaded region denotes a collar neighborhood (in the $\tau$-direction) of $\partial _ v E$ upon which lies the collection of four arcs denoted by $a _ v$.}
\label{fig:chunk_E}
\centering
\end{figure}
At this stage, we have arrived at a partially-defined diffeomorphism
\[
f : s (\widetilde{u}(M)_{2,1})
\dashrightarrow 
s (\widetilde{u}(M)_2)_1 
\]
that is actually a strict contactomorphism if $\varphi ^ {\pm}$ intertwine the Liouville structures on ${\partial _ h } ^ {\pm} s (\widetilde{u}(M) _ {2,1})$ and ${\partial _ h} ^ {\pm}s (\widetilde{u}(M) _ 2) _ 1$, respectively. Assuming this for the moment, we may choose a tight 3-ball $B$ with convex boundary in the interior of $s (\widetilde{u}(M) _ {2,1})$ that contains the subset where $f$ is not defined. Then uniqueness of the tight contact structure on the with convex boundary and connected dividing set can be used to extend $f$.\\ 

\noindent\textit{Step 3.} The diffeomorphisms ${\varphi}^ {\pm}$ may \emph{a priori} fail to intertwine the Liouville structures over the subsurfaces where they are not equal to the identity map, i.e. those not containing ${\partial _ h } ^ {\pm} s (\widetilde{u}(M) _ {2,1}) \setminus (D \times [-\delta , 0] _ s \cup h _ 1)$. To remedy this, let $(\varphi ^ {\pm}) _ {*} (\beta _ {0} ^ {\pm}) = e ^ \tau \theta _ 0$ where $\beta _ {0} ^ {\pm}$ denotes the Liouville form obtained from the restriction of $s (\widetilde{\alpha} _ {2,1})$ to ${\partial _ h }^ {\pm} s (\widetilde{u}(M) _ {2,1})$ and $\theta _ 0$ denotes a volume form on $\Gamma \subset s (\widetilde{u}(M) _ 2) _ 1$ and analogously let $e^{\tau} \theta _ 1$ denote the restriction of $s (\widetilde{\alpha} _ 2) _ 1$ to ${\partial _ h }^ {\pm} s (\widetilde{u}(M) _ 2) _ 1$. Interpolate between $e ^ {\tau} \theta _ 0$ and $e^ {\tau} \theta _ 1$ via a family $e ^ \tau \theta _ \tau$ supported over the region 
$\Gamma \times (-1 , T] _ {\tau} \times [-1 ,1] _ t \subset s (\widetilde{u}(M) _ 2) _ 1$ where $T > 0$ is sufficiently large and we $\dot{\theta}_\tau \ll 0$ so that $Cdt + e ^ \tau \theta _ \tau$ still defines a contact form. Since sutured $\ech$ of 
\[
\big( s(\widetilde{u}(M) _ 2) _ 1 \cup \Gamma \times (-1 , T] _ {\tau} \times [-1 ,1] _ t , 
s(\widetilde{\alpha} _ 2) _ 1 \cup (Cdt + e^ {\tau } \theta _ \tau)
\big)
\]
is canonically isomorphic (by proof of Theorem \ref{thm:suturing_unsutured}) to that without such extension we may assume, without loss of generality, that $(\varphi ^ {\pm}) _ {*} (\beta _ 0 ^ \pm) = \beta _ 1 ^ \pm$ over the vertical neighborhood of the suture of $\Gamma'$ of $s (\widetilde{u}(M) _ 2) _ 1$. 

Finally, to attain $(\varphi ^ {\pm})_{*} (\beta _ 0 ^ {\pm}) = \beta _ 1 ^ {\pm}$ where $\beta _ {1} ^ {\pm}$ denotes the restriction of $s (\widetilde{\alpha} _ 2) _ 1$ to ${\partial _ h } ^{\pm} s (\widetilde{u}(M) _ 2) _ 1$, we adopt the strategy in \cite[Lemma 2.9]{kst} to find an isotopy $\phi _ {\nu}$ compactly supported near ${\partial _ h  }^{\pm} s (\widetilde{u}(M)_2) _ 1$ so that
\begin{itemize}\leftskip-0.25in
    \item $\phi _ {\nu} =id$ whenever 
    $\beta _ {0} ^ {\pm} = \beta _ {1} ^{\pm}$,
    \item $\phi _ 1$ perturbs $s (\widetilde{u}(M) _ 2) _ 1$ within the vertical $t$-extension:
        \[
           ( {\partial _ h }^ + s (\widetilde{u}(M) _ 2) _ 1 ) \times (1 - \varepsilon , \infty) _ {t} 
           \cup
             ( {\partial _ h }^{-} s (\widetilde{u}(M) _ 2) _ 1 ) \times (-\infty , -1 + \varepsilon) _ {t} 
        \]
        to a new sutured contact manifold denoted by $\phi _ 1 s (\widetilde{u}(M) _ 2 ) _ 1$ whose sutured ECH is canonically isomorphic to the original one,
    \item 
    denoting by ${\beta' _ 1 }^ {\pm}$ the restriction of $s( \widetilde{\alpha} _ 2) _ 1$ to the ${\partial _ h} ^ {\pm}$-boundary of $\phi _ 1 s (\widetilde{u}(M) _ 2 ) _ 1$, we have
    \[
   (\phi _ 1 \circ \varphi ^ {\pm}) _ {*} 
   (\beta _ 0 ^ {\pm}) = {\beta' _ 1 }^{\pm}.
    \]
\end{itemize}
As a result, we may as well have assumed $(\varphi ^ {\pm}) _ {*} (\beta _ {0} ^ {\pm}) = \beta _ {1} ^ {\pm}$. 
\end{proof}

\section{Proof of Proposition \ref{prop:handle_surgery}}
\label{sec:proof33}
This section completes the proof of Proposition \ref{prop:handle_surgery}. Our proof has four steps.\\

\noindent\textit{Step 0.} Let $(M , \Gamma , \alpha)$ be a sutured contact manifold with a chosen tailored almost complex structure $J$. For simplicity, suppose that $\Gamma$ is connected and that there is a diffeomorphism $f : R _ + (\Gamma) \rightarrow R _ - (\Gamma) $ so that $f ^ {\big.*} \beta _ + = \beta _ -$. (If not, we can add sufficiently many contact 1-handles away from $\gamma\cap \Gamma$ so as to make $\Gamma$ and keep sutured ECH unchanged.) Then by \cite[Chapter 3]{kst} there is a closed contact manifold $(Y _ n , \alpha _ n)$ that contains $(M , \alpha)$ as a contact submanifold for each $n>0$. Having fixed a symplectization admissible almost complex structure on $\mathbb{R}\times Y_n$ that agrees with $J$ on $\mathbb{R}\times M$, there is a canonical isomorphism
\[
\Phi _ n ^ {L , J} : 
\ech ^ L (M , \Gamma , \alpha , J)
\rightarrow
\ech ^ L (Y _ n , \alpha _ n)
\]
for each $n$ sufficiently large with respect to $L$ (see \cite[Lemma 3.5]{kst} and equation (3.3) therein.) To be brief, the closed manifold $(Y _ n , \alpha _ n)$ indexed by $n$ is obtained from $(M , \Gamma , \alpha)$ by gluing $R _ + (\Gamma)$ to $R _ - (\Gamma)$ via the diffeomorphism $f$ while adding in the `neck' $R _ + (\Gamma) \times [-n , n] _ t$ thus forming a pre-Lagrangian torus boundary \cite[Proposition 2.11]{kst} along which a contact solid torus $\simeq D _ {2 + 2n} \times \Gamma$ is attached in a manner that swaps the meridian and the longitude where $D _ {2 + 2n} = \{ \textbf{x} \in \mathbb{R}^2 :  |\textbf{x}| \leq 2 + 2n\}$. In what follows, we shall refer to the closed contact manifold $(Y _ n, \alpha _ n)$ as the contact closure of $(M,\Gamma,\alpha)$ with neck-length $n$.

The maps $\Phi _ n ^ {L ,J}$'s enjoy the following properties. First, they intertwine inclusion induced maps
            \begin{equation}
            \label{eq:inclusion_plus_J}
            \begin{tikzcd}
                \ech ^ L (M , \Gamma , \alpha , J)
                \rar["\Phi _ n ^ {L , J}"] 
                \dar["i ^ {L , L'} _ J"] 
                & 
                \ech ^ L (Y _ n , \alpha _ n)
                \dar["i ^ {L , L'}"]
                \\
                \ech ^ {L'} (M , \Gamma , \alpha , J) 
                \rar["\Phi _ n ^ {L' , J}"] 
                & 
                \ech ^ {L'} (Y _ n , \alpha _ n) 
            \end{tikzcd}
            \end{equation}
for each $n$ sufficiently large with respect to both $L$ and $L'$ when $L<L'$. Next, notice that for $L>0$ and $n,m>0$ with $n<m$ both sufficiently large with respect to $L$, it follows from \cite[Proposition 3.8]{kst} that there exists an isomorphism 
\begin{equation}
\label{n_to_m_iso_closure}
\Psi ^ L _ {n, m} :
\ech ^ L (Y _ n , \alpha _ n) 
\rightarrow
\ech ^ L (Y _ m , \alpha _ m)
\end{equation}
such that
\[
\Psi ^ L _ {n,m}\circ \Phi _ n ^ {L ,J}
=
\Phi _ m ^ {L , J},
\]
which holds uniformly for different choices of the tailored almost complex structure $J$.\\

\noindent\textit{Step 1.} With the above understood, we construct the map $F ^ L _ {K , J _ 0 J _ 1}$ by making use of the isomorphisms $\Phi _ n ^ {L , J}$ and the $\ech$ cobordism map \cite[Theorem 1.9]{ht2} for closed contact 3-manifolds where $J _ 0$ denotes a tailored almost complex structure for $(M , \Gamma , \alpha)$ while $J _ 1$ denotes a tailored almost complex structure for $(M' , \Gamma , \alpha ')$, the latter being obtained by contact $(+1)$-surgery along $K$ as in Proposition \ref{prop:handle_surgery}.

Indeed, the embedding data used to construct the contact closure $(Y _ n , \alpha _ n)$ for $(M , \Gamma , \alpha)$ can also be used for the surgered manifold $(M' , \Gamma , \alpha')$ as well as for the surgery cobordism -- simply by attaching the symplectic 2-handle on the product $[0,1] _ s \times Y _ n$ since the surgery takes place in the interior of $M$. Denoting the former by $(Y _ n ' , \alpha _ n ')$ and the latter by $(X _ n , \omega _ n)$, the map $F ^ L _ {K , J _ 0 J _ 1}$ is defined in such a way as to make the following diagram commute.
    \[
    \begin{tikzcd}
        \ech ^ L (M , \Gamma , e \alpha , J _ 0 ) 
        \rar["\Phi _ n ^ {L , J _ 0}"]
        \dar[dashed, "F ^ L _ {K , J _ 0 J _ 1}"]
        & 
        \ech ^ L (Y _ n , e \alpha _ n) 
        \dar["\Phi ^ L (X _ n )"]
        \\
        \ech ^ L (M' , \Gamma , \alpha ' , J _ 1 ) 
        \rar[" \Phi _ n ^ {L , J _ 1}"]
        & 
        \ech ^ L (Y _ n ' , \alpha _ n ')
    \end{tikzcd}
    \]
Here $\Phi ^ L (X _ n)$ denotes the $\ech$ cobordism map induced by the exact symplectic cobordism $(X _ n , \omega _ n)$. Notice that \textit{a priori} such definition for $F ^ L _ {K , J _ 0 J _ 1}$ depends on the neck-length of the contact closure(s), or equivalently on $n$, and our goal is to remove such dependency (c.f. \cite[Proposition 4.3]{kst}), a crucial step in the final direct limit argument.\\

\noindent\textit{Step 2.} To remove the \emph{a priori} dependency of $F ^ L _ {K , J _ 0 J _ 1}$ on $n$, we show that the following diagram commutes.
\[
\begin{tikzcd}[row sep=1.2em , column sep = 0.1em]
\ech ^ L (M , \Gamma , e \alpha , J _ 0) \arrow[rr,"\Phi _ n ^ {L , J _ 0}"] \arrow[dr,swap,equal] \arrow[dd,swap,"F ^ L _ {K , J _ 0 J _ 1, n}"] &&
  \ech ^ L (Y _ n , e \alpha _ n) \arrow[dd,swap,"\Phi ^ L (X _ n)"' near start] \arrow[dr,"\Psi ^ L _ {n, m}"] \\
& \ech ^ L (M , \Gamma , e \alpha , J _ 0) \arrow[rr,crossing over,"\Phi ^ {L , J _ 0} _ m" near start] &&
  \ech ^ L (Y _ m , e \alpha _ m) \arrow[dd,"\Phi ^ L (X _ m)"] \\
\ech ^ L (M ' , \Gamma , \alpha ' , J _ 1) \arrow[rr,"\Phi ^ {L , J _ 1} _ n" near end] \arrow[dr,swap,equal] && \ech ^ L (Y _ n ' , \alpha _ n ') \arrow[dr,swap,"\Psi ^ L _ {n, m}"'] \\
& \ech ^ L (M' , \Gamma , \alpha', J _ 1) \arrow[rr,"\Phi ^ {L , J _ 1} _ m"] \arrow[uu,<-,crossing over,"F ^ L _ {K , {J _ 0}' {J _ 1}' , m}" near end]&& \ech ^ L (Y _ m ' , \alpha _ m ')
\end{tikzcd}
\]
In the above diagram, we write $F ^ L _ {K , J _ 0 J _ 1, n}$ to indicate the \emph{a priori} dependence of $F ^ L _ {K , J _ 0 J _ 1}$ on $n$. It suffices to show that the following diagram commutes
\[
    \begin{tikzcd}
        \ech ^ L (Y _ n , e \alpha _ n) 
        \rar["\Psi _ {n, m} ^ L" ] 
        \dar["\Phi ^ L (X _ n )"] 
        & 
        \ech ^ L (Y _ m , e \alpha _ m)  
        \dar["\Phi ^ L (X _ m )"] 
        \\
        \ech ^ L (Y _ n ' , \alpha _ n ') 
        \rar["\Psi _ {n, m} ^ L"] 
        &
        \ech ^ L (Y _ m ' , \alpha _ m ')
    \end{tikzcd}
    \]
where the bottom $\Psi ^ L _ {n, m }$ should be regarded as the version of the isomorphism in \eqref{n_to_m_iso_closure} for $(M ' , \Gamma', \alpha ')$ while the above is for $(M , \Gamma, \alpha)$. 

Indeed, it is proved by pulling all the datum over $X _ m$ back to $X _ n$ via a diffeomorphism $\varphi _ {n ,m} : X _ n \rightarrow X _ m$ that restricts to the identity on $X \subset X_ n, X _ m$ and is constructed between $X _ n \setminus X = [0 ,1] _ s \times (Y _ n \setminus M) $ and $X _ m \setminus X = [0,1] _ s \times (Y _ m \setminus M)$ in an $s$-invariant way via a map (c.f. proof of \cite[Proposition 3.8]{kst})
\[
\varphi _ {n , m} :
Y _ n \setminus M
\rightarrow
Y _ m \setminus M
\]
that is of the form
    \[
    id \times f _ {n , m} : 
    R _ + (\Gamma) \times [-n , n] _ t 
    \rightarrow
    R _ + (\Gamma) \times [- m , m] _ t
    \]
on the neck part in which $f _ {n , m} (t) = \pm (m-n) + t$ whenever $t$ is close to $\pm n$ while on the part of solid tori:
   \[
    \mathfrak{f} _ {n, m} \times id : 
    D  _ {2 + 2n} \times \Gamma \rightarrow 
    D  _ {2 + 2m} \times \Gamma
    \]
we make certain extension from the previous neck part. Such construction in fact can be done in an $\nu$-parametric way with $n \leq \nu \leq m$ by considering a smooth family of embeddings 
\begin{align*}
    f _ {n , \nu} : 
    [-n , n] 
    \rightarrow 
    [-m , m] \\
    \mathfrak{f} _ {n , \nu} :
    D  _ {2 + 2n}
    \rightarrow
    D _ {2 + 2m}
\end{align*}
whose images are respectively $[-\nu , \nu]$ and $D _ {2 + 2\nu}$ thus resulting in a smooth family of diffeomorphisms:
$\varphi _ {n ,\nu} : X _ n \rightarrow X _ {\nu}$. 

Here is an interlude concerning the filtered $\ech$ cobordism map of a symplectization and a slight variant of it. Let $(Y , \lambda)$ be a closed contact 3-manifold with $\lambda$ being $L$-nondegenerate (meaning that all the Reeb orbits with action less than $L$ are nondegenerate) then the $\ech$ cobordism map of the symplectization $([0 , \varepsilon] _ s \times Y , d (e ^s \alpha))$ is the composition (c.f. \cite[Corollary 5.8]{ht2}):
\begin{equation}
    \label{eq:bod_map_of_sym}
    \ech ^ {L} (Y  , e ^ {\varepsilon} \alpha) 
    \xrightarrow{s}
    \ech ^ {e ^ {- \varepsilon}L} (Y , \alpha)
    \xrightarrow{i _ {e ^ {- \varepsilon} L , L }}
    \ech ^ {L} (Y , \alpha)
\end{equation}
in which $s$ denotes the scaling isomorphism and $i _ {e^ {- \varepsilon} L , L}$ is an inclusion induced map. To set the stage for the upcoming statement notice that for a sutured contact manifold $(M , \Gamma , \alpha)$ the smooth family of diffeomorphisms:
\[
\varphi _ {n , s} : 
Y _ n 
\rightarrow 
Y _ s
\]
between contact closures of $(M , \Gamma , \alpha)$ with various neck-lengths $s \in [n ,m]$ defines an isotopy of contact forms on $Y _ n$ by pull-back, i.e. $(Y _ n , \varphi _ {n ,s}  ^ * \alpha _ s) := (Y _ n , a _ s)$. Then according to the proof of \cite[Lemma 4.1]{kst} given $\varepsilon > 0$,
\[
([0 , \varepsilon] _ s \times Y _n , d(e ^ s a _ {\phi (s)})) 
\]
would defines an exact symplectic cobordism as long as $n$ is close enough to $m$ where $\phi : [0 , \varepsilon] _ s \rightarrow [n ,m]$ is a smooth function that equals $n$ (resp. $m$) when $s$ is near to $0$ (resp. $\varepsilon$). The induced cobordism map according to part of the proof of \cite[Proposition 4.3]{kst} is the following composition:
\[
\begin{tikzcd}
    \ech ^ {L} (Y _ n , e ^ {\varepsilon} a _ m) \rar["s"] 
    & 
    \ech ^ {e ^ {- \varepsilon}L} (Y _ n , a _ m) \dar["i _ {e ^ {-\varepsilon}L , L }"] 
    &  
    \\
     & 
     \ech ^ {L} (Y _ n , a _ m) \rar["\Psi _ {m , n} ^ L"'] 
     & 
     \ech ^ {L} (Y _ n , \alpha _ n)
\end{tikzcd}
\]
or equivalently by \eqref{eq:bod_map_of_sym}
\begin{equation}
\label{eq:technical_relation_lower}
\Psi ^ L _ {m , n} \circ \Phi ^ L ([0,\varepsilon] _ s \times Y _ n , d (e ^ s a _ m)).
\end{equation}
There is an analogous statement for the the cobordism
\[
([1 - \varepsilon , 1] _ s \times Y _ n , 
d (e ^s a _ {\phi (1-s)}))
\]
whose induced cobordism map is the composition
\[
    \begin{tikzcd}
    \ech ^ {L} (Y _ n , e  \alpha _ n) 
    \rar["\Psi _ {n , m} ^ L"] 
    & 
    \ech ^ {L} (Y _ n , e \alpha _ m) 
    \dar["s"] 
    &  
    \\
     & 
     \ech ^ {e^ {-1} L} (Y _ n , \alpha _ m) \rar["i _ {e^ {-1} L , L}"'] 
     & 
     \ech ^ {L} (Y _ n , \alpha _ m)
    \end{tikzcd}
    \]
or equivalently
\begin{equation}
\label{eq:technical_relation_upper}
\Phi ^ L ([1 - \varepsilon , 1] _ s \times Y _ n , d (e ^s a _ m)) \circ \Psi ^ L _ {n ,m}.
\end{equation}

To finish the proof we continue to follow the strategy of \cite[Proposition 4.3]{kst}, as in \cite[Lemma 6.5]{ht2}, by interpolating $(X _ n , \varphi _ {n , m} ^ {*} \omega _ m )$ to $(X _ n , \varphi ^ * _ {n , n} \omega _ n = \omega _ n)$ but in a way only near the incoming/outgoing ends and show that the resulting exact symplectic cobordism is 
\begin{itemize}\leftskip-0.25in
    \item 
    homotopic to $(X _ n , \omega _ n)$ thus sharing the same cobordism induced maps by \cite[Theorem 1.9]{ht2},
    \item 
    and conjugates to the cobordism map induced from the original $(X _ n , \varphi ^ * _ {n, m} \omega _ m)$ via the maps $\Psi ^ L _ {n, m}$'s using \eqref{eq:technical_relation_lower} and \eqref{eq:technical_relation_upper}.
\end{itemize}

To spell out the construction: the smooth family of diffeomorphisms 
$\varphi _ {n , \nu} : X _ n \rightarrow X _ {\nu}$ would produce a smooth $\nu$-parameter family of (exact) symplectic forms $\{ \varphi ^ * _ {n , \nu} \omega _ {\nu} \} _ {\nu \in [n , m]}$ over $X _ n$ and can be identified with the symplectizations: 
\[
([0 , \varepsilon] _ s \times Y ' _ n , 
d (e ^ s a _ {\nu} '))
\]
and
\[
([1 - \varepsilon , 1] _ s \times Y _ n ,
d ( e ^ s a _ {\nu}))
\]
respectively near an $\varepsilon$-collar neighborhoods of the incoming/outgoing end in $X _ n$ for all $\nu \in [n ,m]$ in which $a _ {\nu} = \varphi _ {n , \nu} ^ * \alpha _ {\nu}$ and $a _ {\nu} '= \varphi _ {n , \nu} ^ * \alpha _ {\nu} '$. 

We then pick a smooth $\nu$-parameter family of smooth (non-decreasing) functions $\phi _ {\nu} : [0 , \varepsilon] _ s\rightarrow [n,m]$ so that $\phi _ {\nu} = n$ (resp. $m$) when $s$ is near to $0$ (resp. $\varepsilon$) and modify the family $\{ \varphi ^ * _ {n , \nu} \omega _ {\nu}\} _ {\nu \in [n, m]}$ along the above mentioned symplectization parts according to 
\[
([0 , \varepsilon] _ s \times Y ' _ n , 
d (e ^ s a _ {\phi _ {\nu} (s)} '))
\]
and
\[
([1 - \varepsilon , 1] _ s \times Y _ n ,
d ( e ^ s a _ {\phi _ {\nu} (1-s)}))
\]
whose form datum would be symplectic as long as $n<m$ are close enough. To the general case we may proceed step by step : $n < \cdots < n _ i < n _ {i + i} < \cdots < m $ since the key is to make $d \phi _ {\nu} (s) / ds$ small which can be achieved if $n _ {i + 1} - n _ i$ is. The upshot is that this modified family is the promised homotopy of cobordisms and one can show the desired properties following their original strategy.\\

\noindent\textit{Step 3.} In this step we shall see the various maps $F ^ L _ {K , J _ 0 J _ 1}$ are equivariant in the spots of $J _ 0, J _1$ thus collectively define a map
\[
F ^ L _ K : 
\ech ^ L ( M , \Gamma , e \alpha)
\rightarrow
\ech ^ L (M ' , \Gamma , \alpha ')
\]
between respective canonical filtered $\ech$ groups. Recall that in \cite{ht2} it was shown that the filtered $\ech$ groups does not depend on the choice of the almost complex structures used to define them. The corresponding result for the sutured version was proved in \cite[Chapter 3]{kst} where a canonical group $\ech ^ L (M , \Gamma , \alpha)$ is assigned to a given sutured contact manifold and $L > 0$ as the direct limit of a transitive system indexed by all the possible tailored almost complex structures $J$ (see \cite[Theorem 3.10]{kst}) 
\[
\big(  \{ \ech ^ L (M , \Gamma , \alpha , J) \} _ J , \{ \Phi ^ L _ {J , J '}\} _ {J , J '}\big)
\]
and the required equivariance of $F ^ L _ {K , J _ 0 J _ 1}$ can be shown by establishing the following diagram
\begin{equation}
\label{eq:F^L_K_equiv_ with_ J}
\begin{tikzcd}[row sep=1.2em , column sep = 0.1em]
\ech ^ L (M , \Gamma , e\alpha , J _ 0 ) \arrow[rr,"\Phi ^ {L , J _ 0} _ n "] \arrow[dr,swap,"\Phi ^ L _ {J _ 0 {J _ 0}'}"] \arrow[dd,swap,"F ^ L _ {K , J _ 0 J _ 1}"] & &
  \ech ^ L (Y _ n , e \alpha _ n) \arrow[dd,swap,"\Phi ^ L (X _ n)"' near start] \arrow[dr,equal] 
  \\
& \ech ^ L (M , \Gamma , e\alpha , {J _ 0}') \arrow[rr,crossing over,"\Phi ^ {L , {J _ 0}'} _ n" near start] &&
  \ech ^ L (Y _ n , e \alpha _ n) \arrow[dd,"\Phi ^ L (X _ n)"] 
  \\
\ech ^ L (M' , \Gamma , \alpha ' , J _ 1 ) \arrow[rr,"\Phi _ n ^ {L , J _ 1}" near end] \arrow[dr,swap,"\Phi ^ L _ {J _ 1 {J _ 1}'}"] && \ech ^ L (Y _ n ' , \alpha _ n ') \arrow[dr,swap,equal] 
\\
& \ech ^ L (M' , \Gamma , \alpha ' , {J _ 1}') \arrow[rr,"\Phi ^ {L , {J _ 1}'} _ n"] \arrow[uu,<-,crossing over,"F ^ L _ {K , {J _ 0}' {J _ 1}'}" near end]&& \ech ^ L (Y _ n ' , \alpha _ n ')
\end{tikzcd}
\end{equation}
in which the front/back diagram commute because that is the way the maps $F ^ L _ {K , J _ 0 J _ 1}$ ans its prime version are defined; while the top/bottom the way the maps $\Phi ^ L _ {J _ 0 {J _ 0}'}$ and the $J _ 1 {J _ 1}'$-version are defined, see \cite[Lemma 3.9]{kst}. 

To establish the diagram in order for taking the direct limit:
\[
 \begin{tikzcd}
        \ech ^ L (M , \Gamma , e \alpha) 
        \rar["F _ K ^ L "] 
        \dar["i _ {L , L '}"'] 
        & 
        \ech ^ L (M ' , \Gamma , \alpha ' ) 
        \dar["i _ {L , L '}"] 
        \\
        \ech ^ {L'} (M , \Gamma , e \alpha) \rar["F _ K ^ {L'}"'] 
        & 
        \ech ^ {L'} (M ' , \Gamma , \alpha ' )
    \end{tikzcd}
\]
it is enough to show the following diagram for fixed $J _ 0, J _ 1$ commutes:
\begin{equation}
\label{dia:last}
 \begin{tikzcd}
        \ech ^ L (M , \Gamma , e \alpha , J _ 0) 
        \rar["F _ {K , J _ 0 J _ 1}^ L "] 
        \dar["i ^ {L , L '} _ {J _ 0}"'] 
        & 
        \ech ^ L (M ' , \Gamma , \alpha ' , J _ 1) 
        \dar["i ^ {L , L '} _ {J _ 1}"] 
        \\
        \ech ^ {L'} (M , \Gamma , e \alpha , J _ 0) 
        \rar["F _ {K, J _ 0 J _ 1} ^ {L'}"'] 
        & 
        \ech ^ {L'} (M ' , \Gamma , \alpha ' , J _ 1 )
    \end{tikzcd}
\end{equation}
and is equivariant under the maps $\Phi ^ L _ {J _ 0 {J _ 0}'}$, $\Phi ^ L _ {J _ 1 {J _ 1}'}$ and the two $L'$-versions: $\Phi ^ {L'} _ {J _ 0 {J _ 0}'}$, $\Phi ^ {L'} _ {J _ 1 {J _ 1}'}$ in the spots of $J _ 0, J _ 1$. Indeed, \eqref{dia:last} commutes by the way $F ^ L _ {K , J _ 0 J _ 1}$ and $F ^ {L'} _ {K , J _ 0 J _ 1}$ are defined, \eqref{eq:inclusion_plus_J} and the (Inclusion)-part of \cite[Theorem 1.9]{ht2}; while the latter part concerning its equivariance in the spots of $J _ 0, J _ 1$: for $F ^ L _ {K , J _ 0 J _ 1}$ and $F ^ {L'} _ {K , J _ 0 J _ 1}$, they have been proved in \eqref{eq:F^L_K_equiv_ with_ J}; as for $i ^ {L , L '} _ {J _ 0}$ and $i ^ {L , L '} _ {J _ 1}$, see \cite[Theorem 3.10]{kst}.\qed

\newpage
\bibliography{biblio}
\end{document}